\documentclass{amsart}
\usepackage{amssymb,latexsym}
\usepackage{amsmath}
\usepackage{graphicx} 
 \usepackage{color}
\usepackage{enumerate}

\numberwithin{equation}{section}
\theoremstyle{plain}
 \newtheorem{theorem}{Theorem}[section]
 \newtheorem{lemma}[theorem]{Lemma}

\theoremstyle{definition}
 \newtheorem{definition}[theorem]{Definition}
 \newtheorem{notation}[theorem]{Notation}
 
\theoremstyle{remark}
\newcommand \mn [1] {#1^-}
\newcommand \vb {\vec \beta}
\newcommand \mb  {\vec \beta^{-}}

\newcommand \datum{September 5, 2016}
\newcommand \N {\mathbb N}

\newcommand \trn[2] {\{\kern-3pt(#1,#2)\kern-3pt\}^{\kern-0.5pt{\scriptscriptstyle\textup{tr}}}}
\newcommand \Tran {\textup{Tran}}
\newcommand \Equ {\textup{Equ}}
\newcommand \Quo {\textup{Quo}}
\newcommand \Rel {\textup{Rel}}

\newcommand \zero[1] {\Delta_{\kern - 1pt #1}}
\newcommand \quo[2]{\langle #1,#2 \rangle^{\kern-1pt{\scriptscriptstyle\textup{q}}} }
\newcommand \equ[2]{[#1,#2]^{\kern-0.5pt{\scriptscriptstyle\textup{e}}}}

\newcommand \llalpha{\alpha^{\text{\raisebox {1.5pt}{\kern -8pt$\mathord{\Leftarrow}$}}}}
\newcommand \rralpha{\alpha^{\text{\raisebox {1.5pt}{\kern -7pt$\mathord{\Rightarrow}$}}}}
\newcommand \llbeta{\beta^{\text{\raisebox {3.5pt}{\kern -8pt$\mathord{\Leftarrow}$}}}}
\newcommand \rrbeta{\beta^{\text{\raisebox {3.5pt}{\kern -7pt$\mathord{\Rightarrow}$}}}}
\newcommand \llgamma{\gamma^{\text{\raisebox {1.5pt}{\kern -7pt$\mathord{\Leftarrow}$}}}}
\newcommand \rrgamma{\gamma^{\text{\raisebox {1.5pt}{\kern -6pt$\mathord{\Rightarrow}$}}}}
\newcommand \lalpha{\alpha^{\kern -8pt\mathord{\leftarrow}}}
\newcommand \ralpha{\alpha^{\kern -7pt\mathord{\rightarrow}}}
\newcommand \ltau{\tau^{\kern -7pt\mathord{\leftarrow}}}
\newcommand \rtau{\tau^{\kern -6pt\mathord{\rightarrow}}}
\newcommand \lrho{\rho^{\kern -6pt\mathord{\leftarrow}}}
\newcommand \rrho{\rho^{\kern -5pt\mathord{\rightarrow}}}

\newcommand \nothing [1]{}

\newcommand\red[1]{{\textcolor{red}{#1}}}

\newcommand \tbf [1] {\textbf{#1}} 
\newcommand \set[1] {\{#1\}}

\renewcommand\epsilon{{\varepsilon}}

\newcommand \init [1]  {}
\newcommand \tuple [1] {\langle #1\rangle}

\newcommand \lpair [2] {(#1,#2)}
\begin{document}
\title[Small generating sets of $\Quo(A)$ and $\Tran(A)$]
{A concise approach to small generating sets of lattices of quasiorders and transitive relations}
\author[G.\ Cz\'edli]{G\'abor Cz\'edli}
\email{czedli@math.u-szeged.hu}
\urladdr{http://www.math.u-szeged.hu/\textasciitilde{}czedli/}
\address{University of Szeged, Bolyai Institute. 
Szeged, Aradi v\'ertan\'uk tere 1, HUNGARY 6720}
\author[J.\ Kulin]{J\'ulia Kulin}
\email{kulin@math.u-szeged.hu}
\urladdr{http://www.math.u-szeged.hu/\textasciitilde{}kulin/}
\address{University of Szeged, Bolyai Institute. 
Szeged, Aradi v\'ertan\'uk tere 1, HUNGARY 6720}

\thanks{This research was supported by NFSR of Hungary (OTKA), grant number K 115518}
\date{\red{\hfill  \datum}\\
\phantom{nk} 2000 \emph{Mathematics Subject Classification}. Primary 06B99}
\begin{abstract} By H.\ Strietz, 1975, and G.\ Cz\'edli, 1996, the complete lattice $\Equ(A)$ of all equivalences is four-generated, provided the size $|A|$ is an accessible cardinal. Results of I.\ Chajda and G.\ Cz\'edli, 1996, G.\ Tak\'ach, 1996, T.\ Dolgos, 2015, and J.\ Kulin, 2016, show that both the lattice $\Quo(A)$ of all quasiorders on $A$ and, for $|A|\leq \aleph_0$, the lattice $\Tran(A)$ of all transitive relations on $A$ have small generating sets. Based on complicated earlier constructions, we derive some new results in a concise but not self-contained way.
\end{abstract}

\maketitle

\section{Introduction}
\subsection*{Basic concepts}
Quasiorders, also known as preorders, on a set $A$ form a complete lattice $\Quo(A)$. So do the transitive relations on $A$; their complete lattice is denoted by $\Tran(A)$. Similarly, $\Equ(A)$ will stand for the lattice of all equivalences on $A$.
The natural involution, which maps a relation $\rho$ to its inverse, $\rho^\ast:=\rho^{-1}=\set{\lpair x y:\lpair y x\in\rho}$, is an automorhpism of each of the three lattices mentioned above. If, besides arbitrary joins and meets,  the involution is an operation of the structure, then we speak of the \emph{complete involution lattices}  $\Quo(A)$ and $\Tran(A)$.  However, it would not be worth  considering the involution on $\Equ(A)$, because it is the identity map. Unless otherwise stated, \emph{generation} is understood in complete sense. That is, for a subset $X$ of $\Equ(A)$, $\Quo(A)$, or $\Equ(A)$, we say that $X$ generates the complete (involution) lattice in question if the only complete sub lattice (closed with respect to involution) including $X$ is the whole lattice itself.
For $k\in\mathbb N:=\set{1,2,3,\dots}$, we say that a complete lattice $L$ is \emph{$k$-generated} if it can be generated by a $k$-element subset $X$; $k$-generated complete involution lattices are understood similarly.
Since the involution commutes with infinitary lattice terms, we obtain easily that
\begin{equation}
\parbox{8cm}{if a complete involution lattice $L$ is $k$-generated, then the complete lattice we obtain from $L$ by disregarding the involution is $2k$-generated.}
\label{eqtxtdkeGrB}
\end{equation}
Note that when dealing with \emph{finite} sets $A$ or finite lattices, then the adjective ``complete'' is superfluous; this trivial fact will not be repeated all the time later. 

If a complete lattice is generated by 
a four-element subset $X=\set{x_1,x_2,x_3,x_4}$ such that $x_1<x_2$ but both $\set{x_1,x_3,x_4}$ and 
$\set{x_2,x_3,x_4}$ are antichains, the we say that this lattice is \emph{$(1+1+2)$-generated}. 

We need also the concept of \emph{accessible cardinals}. A cardinal $\kappa$ is \emph{accessible} if it is finite, or it is infinite and for every $\lambda\leq \kappa$, 
\begin{itemize}
  \item either $\lambda \leq 2^\mu$ for some cardinal $\mu<\lambda$,
  \item or there is a set $I$ of cardinals such that $\lambda \leq \sum_{\mu \in I}\mu$,  $|I|<\lambda$, and $\mu<\lambda$  for all $\mu\in I$. 
\end{itemize}
Since all sets in this paper will be assumed to be of accessible cardinalities, two remarks are appropriate here. First,  ZFC has a model in which all cardinals are accessible; see Kuratowski~\cite{kuratowski}. Second, we do not have any idea how approach the problem if $|A|$ is an inaccessible cardinal. 

\subsection*{Earlier results}
Since a detailed historical survey has just been given in Cz\'edli~\cite{czgrecent}, here we mention only few known facts. 
By Strietz~\cite{strietz1} and \cite{strietz2}, Z\'adori~\cite{zadori}, and Cz\'edli~\cite{czedlifourgen}, the complete lattice $\Equ(A)$ of all equivalences is four-generated, provided the size $|A|$ of $A$ is an accessible cardinal and $|A|\geq 2$. Also, we know from these papers that $\Equ(A)$ cannot be generated by less than four elements if $|A|\geq 4$. 
We know from Chajda and Cz\'edli~\cite{chcz} and Tak\'ach~\cite{takach} that the complete \emph{involution} lattice $\Quo(A)$ is three-generated for $|A|\geq 2$ accessible, whereby we conclude from \eqref{eqtxtdkeGrB} that $\Quo(A)$ is six-generated as a complete lattice. Actually, we know from  Dolgos~\cite{dolgos} for $2\leq |A|\leq \aleph_0$ and from Kulin~\cite{kulin} for the rest of accessible cardinals that the complete lattice $\Quo(A)$ is five-generated.
Furthermore, it was proved in Cz\'edli~\cite{czgrecent} that the complete lattice $\Quo(A)$ is four-generated for 
$|A|=\set{\aleph_0}\cup (\mathbb N \setminus\set{1,4,6,8,10})$. It is also shown in \cite{czgrecent} that the complete lattice $\Quo(A)$ cannot be generated by less than four elements, provided $|A|\geq 3$.  Special variants of the above results, without considering the lattices $\Equ(A)$ and $\Quo(A)$ complete, were given in Cz\'edli~\cite{czedlismallgen} and \cite{czgrecent}.  Dolgos~\cite{dolgos} has recently shown that the complete lattice $\Tran(A)$ is eight-generated for $2\leq |A|\leq \aleph_0$. Finally, we know from Z\'adori~\cite{zadori}, which improves Strietz~\cite{strietz1,strietz2} by reducing 10 to 7, and from Cz\'edli~\cite{czedlioneonetwo} that the complete lattice $\Equ(A)$ is $(1+1+2)$-generated provided $|A|\geq 7$ and $|A|$ is an accessible cardinal.

 \subsection*{Concise versus self-contained}
Although the proofs given here are short, sometimes very short, these proofs rely on nontrivial earlier constructions. If someone wanted to replace the proofs of the theorems that are included in the rest of the present paper, then he would need to add several additional pages to each of these proofs; typically, about 15-20 pages to the proofs dealing with all accessible cardinals.

\section{Preparatory lemmas and notation}
As it is usual in lattice theory, we use 
``$\subset$'' to denote  proper inclusion, which excludes equality.

\begin{lemma}[{Kulin~\cite[page 61]{kulin}}]\label{lemmakulin}
If $3\leq |A|$ and $S$ is a complete sublattice of $\Quo(A)$ such that $\Equ(A)\subset S$, then $S=\Quo(A)$.
\end{lemma}

\begin{notation} For $a\neq b\in A$, let 
\begin{align*}
\equ a b&:=\set{\lpair a a,  \lpair a b,  \lpair b b, \lpair b a}\cr
\quo a b&:=\set{\lpair a a,  \lpair a b,  \lpair b b},\text{ and}\cr
\trn a b&:=\set{\lpair a b}; \end{align*}
they are the least equivalence, the least quasiorder, and the least transitive relation, respectively,  containing the pair $\lpair a b$. While $\trn a b$ is always an atom of $\Tran(A)$ and all atoms of $\Tran(A)$ are of this form, 
$\equ a b$ is an atom of $\Equ(A)$ iff $\quo a b$ is an atom of $\Quo (A)$ iff $a\neq b$, and all atoms of $\Equ(A)$ and $\Quo(A)$ are of this form. 
\end{notation}

\begin{definition}
By a \emph{Z\'adori configuration} of rank $n\in\mathbb N$, we mean an edge-colored graph $F_n=\set{a_0,a_1,\dots,a_n,b_0,\dots, b_{n-1}}$ with $\alpha$-colored \emph{horizontal edges} $\lpair{a_{i-1}}{a_{i}}$ and $\lpair{b_{j-1}}{b_{j}}$ for
$i\in\set{1,\dots,n}$ and $j\in\set{1,\dots, n-1}$, $\beta$-colored \emph{vertical edges} $\lpair{a_i}{b_i}$ for $i\in\set{0,\dots, n-1}$, and $\gamma$-colored \emph{slanted edges $($of slope $45^\circ)$} 
 $\lpair{a_{i-1}}{b_i}$ for $i\in\set{1,\dots, n}$; these edges are \emph{solid} edges in our figures.  For example, $F_6$ is given in Figure~\ref{fig1} but we have to disregard the dotted edges. 
We do not make a notational distinction between the graph and its vertex set, $F_n$. The colors $\alpha$, $\beta$, and $\gamma$ are also members of $\Equ(F_n)$; we let $\lpair a b\in\alpha$ if there is an $\alpha$-colored path from $a$ to $b$ in the graph, and we define the equivalences $\beta,\gamma\in\Equ(F_n)$ analogously.
\end{definition}

The following lemma is due to Z\'adori~\cite{zadori}. Note that it is implicit in \cite{zadori}, and it was used, implicitly, in Cz\'edli~\cite{czedlismallgen}, \cite{czedlifourgen}, \cite{czedlioneonetwo}, and  \cite{czgrecent}. The lattice operations join and meet are also denoted by $+$ and $\cdot$ (or concatenation), respectively.

\begin{lemma}[Z\'adori~\cite{zadori}]\label{lemmazadori}
If $n\in \mathbb N$ and $A$ is the base set of the Z\'adori configuration $F_n$, then $\Equ(A)$ is generated by 
$\set{\alpha,\beta,\gamma,\equ{a_0}{b_0}, \equ{a_n}{b_{n-1}} }$.
\end{lemma}

The following straightforward lemma was also used, explicitly or implicitly, in several earlier papers; see 
Chajda and Cz\'edli~\cite[second display in page 423]{chcz}, 
Cz\'edli~\cite[circle principle in page 12]{czedlismallgen}, 
\cite[last display in page 55]{czedlifourgen}, 
\cite[first display in page 451]{czedlioneonetwo}, and \cite[Lemma 2.1]{czgrecent},   
Kulin~\cite[Lemma 2.2]{kulin}, Tak\'ach~\cite[page 90]{takach}, and Z\'adori~\cite[second display in page 583]{zadori}.

\begin{lemma}\label{circlemma} For an arbitrary set $A$ and $j,k\in \N$, if $\set{u,v}$, $\set{x_1,\dots, x_{j-1}}$ and $\set{y_1,\dots, y_{k-1}}$ are pairwise disjoint subsets of $A$,  $u=x_0=y_0$, and  $v=x_j=y_k$,  then 
\begin{align*}
\quo{u}{v}&= \Bigl(\sum_{i=1}^j \quo{x_{i-1}}{x_i}\Bigr) \cdot  \Bigl(\sum_{i=1}^k \quo{y_{i-1}}{y_i}\Bigr), \text{ and}\cr
\equ{u}{v}&= \Bigl(\sum_{i=1}^j \equ{x_{i-1}}{x_i}\Bigr) \cdot  \Bigl(\sum_{i=1}^k \equ{y_{i-1}}{y_i}\Bigr)    \text.
\end{align*}
\end{lemma}

\begin{lemma}
\label{lemmatransa}
Assume that $\alpha_1,\dots,\alpha_k\in\Quo(A)$ are \emph{antisymmetric} $($in other words, they are orderings$)$ and $\{\alpha_1$, \dots, $\alpha_k\}$ generates the complete involution lattice $\Quo(A)$. Then $\set{\alpha_1\setminus\zero A,\dots,\alpha_k\setminus\zero A}$ is a generating set of the complete involution lattice $\Tran(A)$. 
The same holds if we consider $\Quo(A)$ and $\Tran(A)$ complete lattices $($without involution$)$.
\end{lemma}

\begin{proof} Let $\Rel(A)$ stand for the complete involution lattice of all binary relations over $A$. The meet in this lattice is the usual intersection, the involution is the map $\rho\mapsto \rho^\ast:=\rho^{-1}$,  
but the join is defined in the following  way: for $\rho_i\in \Rel(A)$ and $(x,y)\in A^2$, we have $(x,y)\in\bigvee\set{\rho_i: i\in I}$ iff there is an $n\in\mathbb N$,  there exists a finite sequence $x=z_0, z_1,\dots, z_n=y$ of elements of $A,$ and there are $i_1,\dots,i_n\in I$ such that $(z_{j-1},z_j)\in \rho_{i_j}$ for all $j\in\set{1,\dots,n}$. Note that $\Tran(A)$ and $\Quo(A)$ are complete involution sublattices of $\Rel(A)$. For a relation $\rho$, denote $\rho\setminus\zero A$ by $\mn\rho$. Instead of $\tuple{\beta_1,\dots,\beta_k}\in \Rel(A)^k$ and 
$\tuple{\mn\beta_1,\dots,\mn\beta_k}$, we write $\vb$ and $\mb$, respectively. We need $k$-ary $|A|$-complete involution lattice terms, which are defined in the usual way by transfinite induction, see, for example, \cite{czginfterm}; these terms are built from at most $|A|$-ary joins and meets and the involution operation $^\ast$. 
%
For such a term $t$,  $\mn t(\vb)$ and $\mn t(\mb)$ will stand for $\mn {(t(\vb))}$ and $\mn {(t(\mb))}$. Then, for every $k$-ary  $|A|$-complete involution lattice term $t$, we have that 
\begin{equation} \text{for every }\vb\in \Rel(A)^k, \quad \mn t(\vb)=\mn t(\mb).
\label{eqfhTnVbWyX}
\end{equation}

If the rank of $t$ is 0, then $t$ is a variable and \eqref{eqfhTnVbWyX} holds obviously. If  \eqref{eqfhTnVbWyX} holds for a term $t$, then it also holds for $t^\ast$, because $^\ast$ is a lattice automorphism. 
Next, assume that $t=\bigwedge\set{t_i:i\in I}$ and \eqref{eqfhTnVbWyX} holds for all the $t_i$. Then  
\begin{align*}\mn t(\vb)&=t(\vb)\setminus\zero A=\bigl(\bigcap\set{t_i(\vb):i\in I}\bigr)\setminus\zero A= \bigcap\set{t_i(\vb)\setminus\zero A:i\in I} 
\cr
&=  \bigcap\set{\mn t_i(\vb): i\in I} 
= \bigcap\set{\mn t_i(\mb): i\in I} \cr
&= \bigcap\set{t_i(\mb)\setminus\zero A: i\in I} = \bigl(\bigcap\set{t_i(\mb): i\in I}\bigr) \setminus\zero A\cr
&=t(\mb) \setminus\zero A = \mn t(\mb),
\end{align*}
whereby \eqref{eqfhTnVbWyX} holds for $t$. 

Next, assume that $t=\bigvee\set{t_i:i\in I}$.  In order to show the validity of \eqref{eqfhTnVbWyX} for $t$, assume first that $(x,y)\in \mn t(\vb)$. Then $x\neq y$ and $(x,y)\in t(\vb)$. So there is \emph{shortest} a finite sequence $x=z_0, z_1,\dots, z_n=y$ of elements of $A$ and there are $i_1,\dots,i_n\in I$ such that $(z_{j-1},z_j)\in t_{i_j}(\vb)$ for all $j\in\set{1,\dots,n}$. 
Since $x\neq y$ and we use a shortest  sequence, $n\in\mathbb N$ is at least 1 and $z_{j-1}\neq z_j$ for $j\in\set{1,\dots,n}$. Thus, $(z_{j-1},z_j)\in \mn t_{i_j}(\vb)$, whereby the induction hypothesis gives that $(z_{j-1},z_j)\in \mn t_{i_j}(\mb) \subseteq t_{i_j}(\mb)$. Therefore, $(x,y)\in t_{i_1}(\mb)\vee\dots\vee t_{i_n}(\mb)\subseteq 
\bigvee\set{t_i(\mb):i\in I}=t(\mb)$. But $x\neq y$, whence $(x,y)\in\mn t(\mb)$. This proves that $\mn t(\vb)\subseteq \mn t(\mb)$.
Conversely, since the lattice operations and the involution are monotone, $t(\mb)\subseteq t(\vb)$. Subtracting $\zero A$, we obtain that $\mn t(\mb)\subseteq \mn t(\vb)$. This proves \eqref{eqfhTnVbWyX}.

Armed with \eqref{eqfhTnVbWyX}, let $a\neq b\in A$. Since $\set{\alpha_1,\dots,\alpha_k}$ generates the complete involution lattice $\Quo(A)$, there is a $k$-ary $|A|$-complete involution lattice term $t$ such that $\quo a b = t(\vec\alpha)$. Subtracting $\zero A$ from both sides, we obtain that $\trn a b=\quo a b\setminus\zero A=t(\vec\alpha)\setminus\zero A=\mn t(\vec\alpha)$. Thus, by \eqref{eqfhTnVbWyX}, $\trn a b=\mn t(\mn{\vec\alpha})$. This means that for all $a\neq b\in A$, the complete involution sublattice $L$ generated by $\mn{\vec\alpha}$ in $\Rel(A)$ contains  $\trn a b$. But $L$ is also what $\mn{\vec\alpha}$ generates in  $\Tran(A)$. Thus,  what we need to prove is that $L=\Tran(A)$. For $a\neq b$, 
$\trn a b\in L$. Based on this containment, for each $c\in A$, we can pick $x,y\in A$ such that $|\set{x,y,c}|=3$; then 
\begin{equation}
\trn c c =(\trn c x \vee \trn x c)\wedge(\trn c y \vee \trn y c)\in L.
\label{eqidprtrmnt}
\end{equation}
Finally, for an arbitrary $\rho\in \Tran (A)$, 
we obtain from $\rho=\bigvee\set{\trn a b: (a,b)\in\rho}$ that $\rho\in L$. Consequently, $L=\Tran(A)$ is generated by $\mn{\vec\alpha}$ as required.
\end{proof}

\section{The lattice on quasiorders}

The following lemmas will lead a theorem.

\begin{lemma} For a set $A$ such that $13\leq |A|<\aleph_0$ and $|A|$ is odd, 
$\Quo(A)$ is $(1+1+2)$-generated.
\end{lemma}

\begin{figure}[htb]
\centerline
{\includegraphics[scale=1.0]{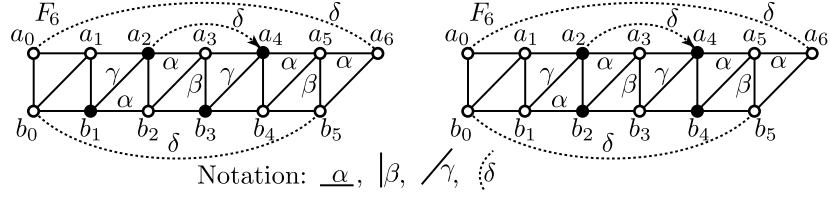}}
\caption{$F_6$ with dotted $\delta$-edges, twice}\label{fig1}
\end{figure}

\begin{proof} Take $F_n$ for $6\leq n\in\mathbb N$ from 
Lemma~\ref{lemmazadori}, see Figure~\ref{fig1}. Define  
\begin{equation}
\delta=\equ{a_0}{a_n}+\equ{b_0}{b_{n-1}}+\quo{a_2}{a_4}\in\Quo(A);
\label{eqdhfhBezv}
\end{equation}
the join above is denoted by plus and it is taken in $\Quo(A)$. Note that \eqref{eqdhfhBezv} makes sense since, say, $\equ{a_0}{a_n}\in\Equ(A)\subseteq\Quo(A)$. In the figure, $\delta$ is visualized by the dotted lines.  Let $L:=[\alpha,\dots,\delta]\leq \Quo(A)$. The $(\delta+\delta^{-1}+\gamma)$-block of $a_2$ 
is $\set{b_1,a_2,b_3,a_4}$, see the black-filled elements on the left, whereby it follows easily that $\equ{a_0}{b_0} = \beta(\gamma+\delta)$.
Similarly, the $(\delta+\delta^{-1}+\beta)$-block of $a_2$ consists of the 
 black-filled elements on the right, and we conclude that $\equ{a_n}{b_{n-1}}=\gamma(\beta+\delta)$. By Lemma~\ref{lemmazadori}, $\Equ(A)\subseteq L$. Actually, $\Equ(A)\subset L$, since $\delta\in L\setminus\Equ(A)$. Thus, the statement follows from Lemma~\ref{lemmakulin}.
\end{proof}

Let us agree that every infinite cardinal is even.

\begin{lemma}\label{lemmatvnylc}
For $58\leq |A|\leq \aleph_0$, if $|A|$ is even, then the complete lattice 
$\Quo(A)$ is $(1+1+2)$-generated. 
\end{lemma}

\begin{figure}[htb]
\centerline
{\includegraphics[scale=1.0]{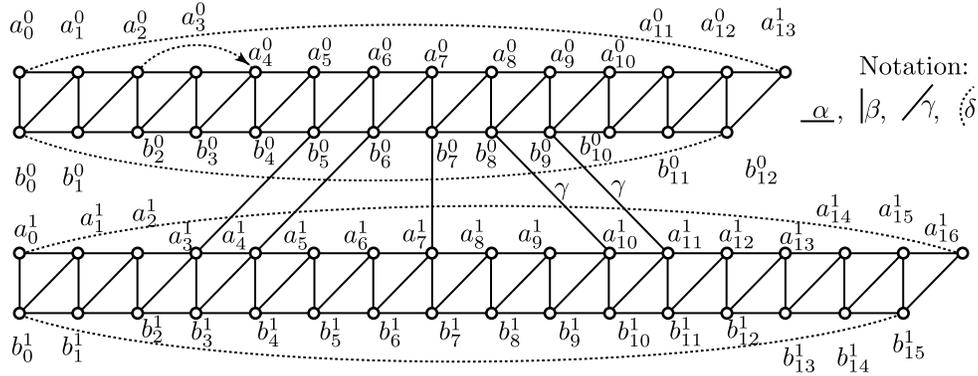}}
\caption{$F_{13}\oplus F_{13}$ }\label{fig2}
\end{figure}
%
%

\begin{proof}For $13<t\in\N$, define the graph $F_{13}\oplus F_t$ in the same way (but with a new notation) as 
in Cz\'edli~\cite{czedlismallgen}; see Figure~\ref{fig2} for $t=16$.  Note that, for example, $\lpair{b^0_9}{a^1_{11}}$ is a $\gamma$-colored edge, no matter how large $t$ is. Let $A:=F_{13}\oplus F_t$.
The dotted lines stand for $\delta$ again; note that because of $\lpair{a^0_2}{a^0_4}\in\delta$ but $\lpair{a^0_4}{a^0_2}\notin\delta$, $\delta\notin\Equ(A)$. Let  $L:=[\alpha,\dots,\delta]\leq \Quo(A)$.
Clearly, $|A|=2\cdot 13+1 + 2t+1$ ranges in $\set{58,60, 62,\dots}\subset \mathbb N$.  For $\aleph_0$, we let  $A:=F_{13}\oplus F_{14}\oplus F_{15}\oplus \dots$ as in \cite{czedlismallgen}.  Since the $\delta$-edge $(a^0_2,a^0_4)$ does not disturb anything in the proof given in \cite{czedlismallgen}, $\Equ(A)\subseteq L$. This inclusion, $\delta\in L\setminus \Equ(A)$, and Lemma~\ref{lemmakulin} yields the lemma. Also, by the argument of \cite{czedlismallgen},
\begin{equation}
\parbox{7.5cm}{the sublattice (not a complete one!) generated by $\set{\alpha,\beta,\gamma,\delta}$ contains all atoms of $\Quo(A)$.}\qedhere
\end{equation}
\label{eqtxtghBuEbG}
\end{proof}

Next, we formulate the ``large accessible'' counterpart of Lemma~\ref{lemmatvnylc}.

\begin{lemma} If $\,\aleph_0\leq |A|$ is accessible, then $\Quo(A)$ is $(1+1+2)$-generated.
\end{lemma}

\begin{proof} Instead of $F_{29}$ in Cz\'edli~\cite[Figure 1]{czedlioneonetwo}, start with $F_{34}$. Instead of the four switches of $F_{29}$, designate five switches in $F_{34}$ but use only four of them exactly in the same way as in \cite{czedlioneonetwo}. Follow the construction of \cite{czedlioneonetwo} with $F_{34}$ instead of $F_{29}$ and, of course, not using the fifth switch. 
This change does not disturb the argument, and we obtain  a $(1+1+2)$-generating set of  the complete lattice $\Equ(A)$; the only difference is that  very many unused switches remains by the end of the construction. 

Now, we pick one of the unused switches and  turn it to the, say, upper half of  \cite[Figure 4]{czedlioneonetwo} but in a slightly modified form:  instead of the non-oriented dotted arc (for $\delta$), now we use an oriented arc. 
Since this arc changes neither $\beta(\gamma+\delta)$, nor   $\gamma(\beta+\delta)$,  $\delta\notin \Equ(A)$, we still have that $\Equ(A)\subseteq [\alpha,\dots,\delta]$. This fact, $\delta\notin\Equ(A)$ and Lemma~\ref{lemmakulin} complete the proof.
\end{proof}

The following lemma adds $6$, $8$, and $10$ to the scope of the main result of Cz\'edli \cite{czgrecent}; unfortunately, the case $|A|=4$ remains unsettled. Furthermore, it simplifies the approach of \cite{czgrecent} for finite sets $A$ with  $|A|$ even.

\begin{lemma}\label{htoLkd} For $6\leq |A|\in\N$ even, the $($complete$)$ lattice $\Quo(A)$ is four-generated.
\end{lemma}

\begin{figure}[htb]
\centerline
{\includegraphics[scale=1.0]{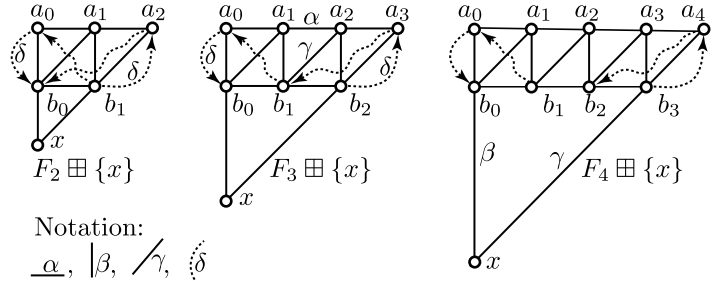}}
\caption{$F_n\boxplus\set{x}$  for $n\in\set{2,3,4}$}\label{fig3}
\end{figure}

\begin{proof} For $n\in\set{6,8,10,12,\dots}$, 
in accordance with our previous constructs and notation,
take the  \emph{one-point extension} $A:=F_n\boxplus\set{x}$ of $F_n$; see Figure~\ref{fig3} for $n\in\set{6,8,10}$. Let $L:=[\alpha,\dots,\delta]$. Also, 
let $A':=A\setminus\set x$, and let $\Quo'(A):=\set{\mu\in \Quo(A): \text{ the } \mu\text{-block of }x \text{ is } \set x }$. For $\epsilon\in\Quo(A)$, let $\epsilon':=\epsilon(\alpha+\delta)\in\Quo'(A)$. By Cz\'edli~\cite{czgrecent} and $\Quo'(A)\cong\Quo(A')$, $\Quo'(A)\subseteq L$. Clearly, we have that $\equ {a_0}x=\beta(\equ{a_0}{a_n}+\gamma)$ and $\equ {a_n}x=\gamma(\equ{a_0}{a_n}+\beta)$ belong to $L$. Hence, Lemma~\ref{circlemma} gives that $\Equ(A)\subseteq L$. Thus, Lemma~\ref{lemmakulin} applies.
\end{proof}

Now, the conclusion of this section is summarized in the following theorem.

\begin{theorem} Let $A$ be a non-singleton set with accessible cardinality.  Then the following statements hold.
\begin{itemize}
  \item If $|A|\neq 4$, then the complete lattice $\Quo(A)$ is four-generated.
 \item If $|A|\geq 13$ and either $|A|$ is an odd number, or $|A|\geq 58$ is even, then the complete lattice $\Quo(A)$ is $(1+1+2)$-generated. 
 \item If $13\leq |A|\leq \aleph_0$ and either $|A|$ is an odd number, or $|A|\geq 58$ is even, then lattice $\Quo(A)$ $($not a complete one now$)$ contains a $(1+1+2)$-generated sublattice that includes all atoms of $\Quo(A)$.
\end{itemize}
\end{theorem}

\section{The complete lattice of transitive relations}

\begin{lemma}\label{lemmakjonway} If $\,3\leq |A|$ and $|A|$ is an accessible cardinal, then the complete lattice $\Tran(A)$ is six-generated.
\end{lemma}

\begin{proof} By Cz\'edli~\cite{czedlifourgen}, there are $\alpha_1,\dots,\alpha_4\in \Equ(A)$ such that
$\set{\alpha_1,\dots,\alpha_4}$. Let $\rho$ be a strict linear order on $A$; for example, it can be a well-ordering. In order to see that 
the complete sublattice $L:=[\alpha_1,\dots,\alpha_4,\rho,\rho^{-1}]$ is actually $\Tran(A)$; it suffices to show that $L$ contains all the atoms of $\Tran(A)$. Take an atom; it is of the form $\trn a b$.  First, assume that $a\neq b$. Then either $\rho$, or $\rho^{-1}$ contains the pair $(a,b)$. Hence, 
$\trn a b$ is either $\equ a b\wedge\rho$, or $\equ a b\wedge\rho^{-1}$. In both cases, since $\equ a b\in \Equ(A)=[\alpha_1,\dots,\alpha_4]\subseteq L$, we obtain that $\trn a b\in L$. Second, assume that $a=b$; that is, we need to deal with $\trn a a$. The assumption $3\leq |A|$ allows us to pick $x,y\in A$ such that $|\set{a,x,y}|=3$. Using \eqref{eqidprtrmnt} with $a$ in place of $c$, we obtain that $\trn a a\in L$, as required.
\end{proof}

\begin{lemma}\label{lemmardhGT} If $\,3\leq |A|$ and $|A|$ is and accessible cardinal, then the complete involution lattice $\Tran(A)$ is  three-generated.
\end{lemma}

\begin{proof}
Observe that the three generators constructed in Tak\'ach~\cite{takach} are orderings. Thus, Lemma~\ref{lemmatransa} applies.
\end{proof}

Note that this proof is more complicated than the proof of Lemma~\ref{lemmakjonway}, because this proof uses Lemma~\ref{lemmatransa}. Note also that \eqref{eqtxtdkeGrB} and Lemma~\ref{lemmardhGT}  imply Lemma~\ref{lemmakjonway}.
Now, based on Lemmas~\ref{lemmakjonway} and \ref{lemmardhGT}, we are in the position to conclude this section and the paper with the following theorem.

\begin{theorem}
If $A$ is a set such that $\,3\leq |A|$ and $|A|$ is an accessible cardinal,
then $\Tran(A)$ is six-generated as a complete lattice, and it is three-generated as a complete involution lattice.
\end{theorem}

%
%


\begin{thebibliography}{99}

%
%
\bibitem{chcz}
  Chajda, I., Cz\'edli, G. (1996).
  How to generate the involution lattice of quasiorders? \emph{Studia Sci. Math. Hungar.} \tbf{32}: 415--427.

\bibitem{czginfterm}
   Cz\'edli, G. (1981).
   On properties of rings that can be characterized by infinite lattice identities.   
  \emph{Studia Sci. Math. Hungar.}, \tbf{16}:  45--60.

\bibitem{czedlismallgen}
   Cz\'edli, G. (1996).
   Lattice generation of small equivalences of a countable set. \emph{Order} \tbf{13}: 11--16.

\bibitem{czedlifourgen}
   Cz\'edli, G. (1996).
   Four-generated large equivalence lattices. \emph{Acta Sci. Math. (Szeged)} \tbf{62}: 47--69.


\bibitem{czedlioneonetwo}
   Cz\'edli, G. (1999).
   (1+1+2)-generated equivalence lattices. \emph{J. Algebra} \tbf{221}: 439--462.

\bibitem{czgrecent}
 Cz\'edli, G. (revised version of August 13, 2016).
 Four-generated quasiorder lattices and their atoms in a four-generated sublattice. \emph{Communications in Algebra}, submitted.

\bibitem{dolgos}
  Dolgos, T. (2015).
  Generating equivalence and quasiorder lattices over finite sets.  \emph{BSc Theses, University of Szeged} (in Hungarian).

\bibitem{kulin}
  Kulin, J. (2016).
  Quasiorder lattices are five-generated.
  \emph{ Discuss. Math. Gen. Algebra Appl.} \tbf{36}: 59--70. 


\bibitem{kuratowski}
  Kuratowsk, K. (1925).
  Sur l\`etat actuel de laxiomatique de la th\'eorie des ensembles, Ann. Soc. Math. Polon. \tbf{3}, 146--147.

\bibitem{strietz1}
   Strietz H. (1975).
   Finite partition lattices are four-generated. In: \emph{Proc. Lattice Th. Conf. Ulm, 1975}, pp.\ 257--259.

\bibitem{strietz2}
   Strietz H. (1977).
   \"Uber Erzeugendenmengen endlicher Partitionverb\"ande.
   \emph{Studia Sci. Math. Hungarica} \tbf{12}:1--17.

\bibitem{takach}
  Tak\'ach, G. (1996): 
  Three-generated quasiorder lattices. 
  \emph{Discuss. Math. Algebra Stochastic Methods} \tbf{16}: 81--98. 

%

\bibitem{zadori}
Z\'adori, L. (1986)
 Generation of finite partition lattices.  In: \emph{Lectures in universal algebra. 
 (Proc. Colloq. Szeged, 1983)}
Colloq. Math. Soc. J\'anos Bolyai, Vol. 43. 
Amsterdam: North-Holland, pp.  573--586. 

%
 
 
\end{thebibliography}
\end{document}